\documentclass[10pt]{amsart}

\usepackage{amssymb}
\usepackage{mathtools}
\usepackage{graphicx}
\usepackage[hidelinks]{hyperref}
\usepackage{stmaryrd}
\usepackage{amsaddr}

\usepackage[margin=1in]{geometry}

\theoremstyle{plain}
\newtheorem{theorem}{Theorem}
\newtheorem{proposition}{{Proposition}}
\newtheorem{lemma}{Lemma}[section]

\numberwithin{equation}{section}

\newcommand{\eps}{\epsilon}
\newcommand{\rt}{\circ}
\newcommand{\umn}{(G,\circ)}

\newcommand{\dg}{\mathrm{deg}(\circ)}
\newcommand{\G}{\mathcal{G}^{*}}
\newcommand{\ka}{\kappa}

\newcommand{\pr}[1]{\mathbf{Pr}\left [ #1 \right]}
\newcommand{\E}[1]{\mathbf{E}\left [ #1 \right ]}

\newcommand{\ind}[1]{\mathbf{1}_{\{#1\}}}
\newcommand{\nb}[1]{\mathrm{NBW}(#1)}
\newcommand{\cev}[1]{\overset{\shortleftarrow}{#1}}

\title{A lower bound on the spectrum of unimodular networks}

\author{Mustazee Rahman}
\address{Department of Mathematical Sciences, Durham University, UK.}
\email{mustazee@gmail.com}

\date{}

\begin{document}

\maketitle

\begin{abstract}
\normalsize
Unimodular networks are a generalization of finite graphs in a stochastic sense.
We prove a lower bound to the spectral radius of the adjacency
operator and of the Markov operator of an unimodular network in terms of its
average degree. This allows to prove an Alon-Boppana type bound for the
largest eigenvalues in absolute value of large, connected, bounded degree graphs,
which generalizes the Alon-Boppana theorem for regular graphs.

A key step is establishing a lower bound to the spectral radius of an unimodular tree
in terms of its average degree. Similarly, we provide a lower bound on the volume
growth rate of an unimodular tree in terms of its average degree.
\end{abstract}

\section{Introduction} \label{sec:intro}

The Alon-Boppana theorem \cite{AB} states that if $G_n$ is a sequence of
finite, connected, $d$-regular graphs with $|G_n| \to \infty$ then the second largest
eigenvalue of the adjacency matrix of $G_n$ in absolute value, say $\sigma_2(G_n)$,
satisfies $\liminf_n \, \sigma_2(G_n) \geq 2\sqrt{d-1}$. The quantity $2 \sqrt{d-1}$ is the
spectral radius of the $d$-regular tree, which represents the exponential growth rate of the number of
closed walks in the $d$-regular tree around a fixed vertex. Another version of the
theorem by Serre \cite{Serre} states that for any $\eps > 0$, there is a positive constant $c(\eps,d)$
such that any finite $d$-regular graph has at least $c(\eps,d)$-proportion of
its eigenvalues having absolute value larger than $2\sqrt{d-1} - \eps$.

What can be said about these types of spectral lower bounds for non-regular or even
infinite graphs? This paper provides such bounds for unimodular networks, a generalization
of finite graphs to a stochastic setting. Using the framework of local convergence of graphs
this provides lower bounds to the top eigenvalues in absolute value of finite, bounded degree
graphs in terms of their average degree. Before stating these results we explain some of the background.

Greenberg \cite{Greenberg} extended the aforementioned theorem of Serre to arbitrary finite graphs,
proving the following. Let $G$ be a locally finite, connected graph with a countable number of vertices.
Let $W_k(G,x)$ be the set of closed walks in $G$ of length $k$ starting from a vertex $x$.
Its spectral radius $\rho(G)$ is the operator norm of the adjacency matrix of $G$ acting on $\ell^2(G)$.
Greenberg proved that for any tree $T$ and any $\eps > 0$, there is a constant $c(\eps,T) > 0$
such that if a finite graph $G$ has universal cover $T$ then at least $c(\eps,T)$-proportion of its eigenvalues
have absolute value at least $\rho(T) - \eps$. (See \cite{Lubotzky} where the result is stated as well.)
Various strengthenings of the Alon-Boppana theorem have later been proved in \cite{Cioaba, Mohar},
and \cite{Elek} gives a Cheeger bound for graph Laplacians.

Afterwards, Hoory \cite{Hoory} proved that if $G$
is a finite graph with $m$ edges that is not a tree, and $T$ is its universal
cover, then $\rho(T) \geq 2 \sqrt{\Lambda}$ where
$\Lambda = \prod_{v \in G} (\deg(v) -1)^{\deg(v)/2m}$. It can be shown that
$\Lambda \geq d_{\mathrm{av}}(G)-1$, where $d_{\mathrm{av}}(G)$ is
the average degree of $G$. Combining Greenberg's theorem with Hoory's implies
that the set of finite and connected graphs sharing a common universal cover
$T$ has the property that for any $\eps > 0$, any graph $G$ from this set
has at least $c(\eps,T)|G|$ eigenvalues with absolute value at least
$2\sqrt{d_{\mathrm{av}}(G)-1} - \eps$.

Sharing a common universal cover is a form of spatial homogeneity for graphs.
Indeed, if two finite graphs have a common universal cover then they also have a
common finite cover \cite{Leigh}. This implies, for instance, that both graphs have the
same spectral radius, average degree, and also the same degree distribution.
In order to prove Alon-Boppana type bounds it is necessary to have some form of
spatial homogeneity. As an example, if the complete graph on $n$-vertices is glued to a path
of length $n$ at a common vertex then the average degree of the resulting graph is at least
$n/2$ while all but the largest eigenvalues have absolute value at most 2.

We consider a stochastic form of spatial homogeneity whereby graphs look homogenous around most vertices.
This is the notion of \emph{unimodular networks}. Roughly speaking, a unimodular network is a
random rooted graph, possibly infinite, that is homogeneous in the sense that shifting the
root to its neighbour does not change the distribution; Section \ref{sec:unimodular} contains the definition.
Finite connected graphs with a uniform random choice of root are unimodular.
Several examples and a rather thorough discussion about unimodular networks may also
be found in \cite{AL} and references therein.

Under natural assumptions, we prove the spectral radius of a unimodular network
is at least $2\sqrt{d_{\mathrm{av}}-1}$, where $d_{\mathrm{av}}$ is the expected
degree of the root. A similar lower bound is proved for the spectral radius of its
simple random walk, which seems to be new even for large finite graphs.
As a consequence of these bounds one finds an analogue of Serre's theorem for
the adjacency and Markov operators of unimodular networks.

We also derive Alon-Boppana type bounds for the eigenvalues of the adjacency matrix and
of the simple random walk (Markov operator) for any growing sequence of connected, bounded degree graphs.
Regarding the adjacency matrix, suppose $G_n$ is a sequence of finite, connected, bounded degree graphs
with size $|G_n| \to \infty$. Then the $j$-th largest eigenvalue of $G_n$ in absolute value, say $\sigma_j(G_n)$,
satisfies $\liminf_n\, \sigma_j(G_n) \geq \liminf_n \, 2\sqrt{d_{\mathrm{av}}(G_n)-1}$.

We also prove that the volume growth rate of a unimodular tree with no leaves is at least
$d_{\mathrm{av}}-1$, where $d_{\mathrm{av}}$ is the expected degree of the root.
This in turn provides a lower bound on the growth rate of non-backtracking walks in certain unimodular networks.

\subsection{Unimodular network and mass transport principle} \label{sec:unimodular}

Let $(G,x)$ be a rooted graph where the distinguished vertex $x$ is the root,
$G$ is locally finite, has a countable number of vertices and is connected.
Two such rooted graphs are isomorphic if there is a graph isomorphism between them that
takes the root of one graph to the other's. Let $\G$ be the set of isomorphism
classes of such rooted graphs. The distance between $(G, x), (H,y) \in \G$ may be defined as
$1/(1+R)$, where $R = \min \{r: B_r(G,x) \cong B_r(H,y)\}$ and $B_r(G,x)$ is the
$r$-neighbourhood of $x$ in $G$. With this distance, $\G$ is a Polish space.
A \emph{random rooted graph} is a Borel probability measure on $\G$, which is conveniently
realized as a $\G$-valued random variable.

A random rooted graph $\umn$ is a \emph{unimodular network} if
\begin{equation} \label{eqn:MTP}
\E{ \sum_{x \in V(G)} f(G,\circ,x)} = \E{ \sum_{x \in V(G)} f(G,x,\circ)}\end{equation}
for every non-negative and measurable function $f$ defined on the set of isomorphism classes
of doubly rooted graphs $(G,x,y)$. Equation (\ref{eqn:MTP}) is called the mass transport principle.
To verify unimodularity it suffices that the mass transport principle holds only for those $f$ that satisfy
$f(G,x,y) = 0$ if $x$ and $y$ are not neighbours in $G$; see \cite[Proposition 2.2]{AL}.
\smallskip

\paragraph{\textbf{Examples.}}
A finite graph $G$ rooted at a uniformly random vertex $\circ$ of $G$ is a unimodular network.
The Cayley graph of any finitely generated group, rooted at its identity, is a deterministic unimodular network.
So the lattices $\mathbb{Z}, \mathbb{Z}^2, \ldots$ are unimodular networks,
as are the infinite regular trees $\mathbb{T}_3, \mathbb{T}_4$, etc. Examples of unimodular
trees include periodic trees, Poisson-Galton-Watson trees, and more generally, unimodular
Galton-Watson trees \cite[Examples 1.1 and 10.2]{AL}.
\smallskip

\paragraph{\textbf{Local convergence.}}
The space of random rooted graphs carries the topology of
weak convergence: $(G_n,\rt_n)$ converges to $\umn$ if
$$\E{f(G_n,\rt_n)} \to \E{f(\umn)}$$ for every
bounded and continuous $f : \G \to \mathbb{R}$.
Restricted to unimodular networks, this provides the natural notion of convergence.
The limit of a sequence of unimodular networks is also a unimodular network; see \cite[Lemma 2.1]{Bor}.
This notion of convergence of unimodular networks, especially for finite graphs rooted uniformly at random,
is called local convergence or also Benjamini-Schramm convergence as they formulated the concept \cite{BS}.


\paragraph{\textbf{Spectral radius.}}
Recall that $W_k(G,x)$ is the set of closed walks in $G$ of length $k$ starting from $x$.
The spectral radius of a unimodular network $\umn$ is defined to be
$$ \rho(G) = \lim_{k \to \infty} \,\E{\,|W_{2k}(G,\rt)|\,}^{\frac{1}{2k}}.$$
The quantity $\E{|W_k(G,\rt)|}$ is in fact the $k$-th moment of a Borel probability
measure of $\mathbb{R}$ called the spectral measure of $\umn$, as explained further in Section \ref{sec:pre}.
The spectral radius is then the largest element in absolute value in the support of the spectral measure.
If $G$ is a finite graph then its spectral measure is the empirical measure of the eigenvalues
of its adjacency matrix.

Similarly, we can define the spectral measure and spectral radius of the simple random walk (SRW)
on $\umn$. For $(G,x) \in \G$, let $p_{k}(G,x)$ be the $k$-step return probability of the SRW
on $(G,x)$ started from vertex $x$. The spectral radius of the SRW on a unimodular network $\umn$ is
$$\rho_{\mathrm{SRW}}(G) = \lim_{k \to \infty} \, \E{p_{2k}(G,\rt)}^{\frac{1}{2k}}.$$
\smallskip

\paragraph{\textbf{Universal cover.}}
The universal cover $T_G$ of a connected, locally finite graph $G$ is the unique tree for which
there is a surjective graph homomorphism $\pi : T_G \to G$, called cover map, such that $\pi$
is an isomorphism on the 1-neighbourhood of every vertex. For $(G,x) \in \G$, let $(T_G,\hat{x})$
be its universal cover rooted at any $\hat{x}$ such that $\pi(\hat{x}) = x$
(all such $(T_G,\hat{x})$ have the same rooted isomorphism class). The cover map sends
closed walks in $T_G$ starting from $\hat{x}$ to closed walks in $G$ from $x$ in an injective manner.
Thus, $\rho(G) \geq \rho(T_G)$. The SRW on $(G,x)$ is the projection of the SRW on $(T_G,\hat{x})$
by the cover map. Therefore, $\rho_{\mathrm{SRW}}(G) \geq \rho_{\mathrm{SRW}}(T_G)$.
If $(G,\rt)$ is a unimodular network then its universal cover tree $(T_G,\hat{\rt})$
is also unimodular. Here, $(T_G,\hat{\rt})$ is constructed for every sample outcome of $\umn$.

\subsection{Statement of results} \label{sec:theorems}

\begin{theorem} \label{thm:cover}
Let $(T,\rt)$ be a unimodular tree with $\E{\dg} < \infty$ and no leaves almost surely. Then
$$\rho(T) \geq 2 \,\exp{\left \{\frac{\E{\dg \log (\sqrt{\dg -1})}}{\E{\dg}} \right \}} \geq 2 \sqrt{\E{\dg} - 1}\,.$$
Additionally, if $(T,\rt)$ has deterministically bounded degree then
$$\rho_{\mathrm{SRW}}(T) \geq 2\,\exp{\left \{\frac{\E{\dg \log \Big (\frac{\sqrt{\dg -1}}{\dg} \Big)}}{\E{\dg}} \right \}}
\geq \frac{2 \,\E{\dg} \sqrt{\E{\dg}-1}}{\E{\dg^2}}.$$
\end{theorem}

The following theorem is about the spectrum of the adjacency operator of unimodular networks and finite graphs.
\begin{theorem} \label{thm:spec}
I) Unimodular networks: Let $(G_n,\rt)$ be a sequence of unimodular networks such that $(G_n,\rt) \to \umn$ locally.
Suppose that $\rho(G) < \infty$. Let $(T_G,\rt)$ be the universal cover of $\umn$. Let $\mu_n$ denote the spectral measure of
$(G_n,\rt)$ and let $\mu_{T_G}$ denote it for $(T_G,\rt)$.

For every $\eps > 0$, there is a constant $c(\eps, \rho(G), \rho(T_G)) > 0$ such that
$$\liminf_{n \to \infty} \, \mu_n \left (\{|x| > \rho(T_G) - \eps\} \right) \geq c(\eps, \rho(G), \rho(T_G)).$$

II) Finite graphs: Let $G_n$ be a sequence of finite, connected graphs with vertex
degrees bounded by $\Delta$ and $|G_n| \to \infty$. Let $\sigma_j(G_n)$ be the
$j$-th largest eigenvalue in absolute value of the adjacency matrix of $G_n$, counted with multiplicity;
these are the singular values of $G_n$. Let $d_{\mathrm{av}}(G_n)$ denote the average degree of $G_n$.

For every $j \geq 1$,
$$ \liminf_{n \to \infty} \, \sigma_j(G_n) \geq \liminf_{n \to \infty} \, 2 \sqrt{d_{\mathrm{av}}(G_n) -1}\,.$$
\end{theorem}

The next theorem is about the spectrum of simple random walk.
For a finite and connected graph $G$, its 2-core, denoted $G^{\mathrm{core}}$, is the subgraph obtained by
iteratively removing leaves from $G$ until none remains. So if $G$ has no leaves then $G^{\mathrm{core}}$ equals $G$,
whereas if $G$ is a tree then $G^{\mathrm{core}}$ is empty.

\begin{theorem} \label{thm:srwspec}
Let $G_n$ be a sequence of finite and connected graphs with all vertex degrees at most $\Delta$.
Let $\mu^{\mathrm{SRW}}_{G_n}$ denote the empirical measure of the eigenvalues of the Markov
operator of $G_n$, that is, of the matrix $P$ with entries $P(x,y) = \frac{1}{\deg{x}} \ind{x \sim y}$ for $x,y \in V(G_n)$.

Suppose $|G_n| \to \infty$ and $|G_n^{\mathrm{core}}| / |G_n| \to 1$. Then for every $\eps > 0$,
$$ \liminf_{n \to \infty} \, \mu^{\mathrm{SRW}}_{G_n} \left ( \left \{ |x| >
\frac{2 \sqrt{d_{\rm{av}}(G_n)-1}}{\frac{1}{2 |E(G_n)|} \sum_{x \in G_n} (\deg{x})^2} - \eps \right \} \right) > 0.$$
\end{theorem}
Note that $\frac{1}{2 |E(G_n)|} \sum_{x \in G} (\deg{x})^2$ is the average degree of $G$
with respect to the stationary measure of its simple random walk, which assigns probability $\frac{\deg{x}}{2|E(G)|}$
to a vertex $x$.

The final theorem is about volume growth.
\begin{theorem} \label{thm:vol}
Let $(T,\rt)$ be a unimodular tree with $\E{\dg} < \infty$ and with no leaves almost surely.
Let $S_r(T,\rt) =  \{x \in V(T): \mathrm{dist}_T(o,x) = r\}$. Then,
\begin{align*}
\E{ |S_r(T,\rt)|} &\geq \E{\dg} \cdot \exp \left \{(r-1) \frac{\E{\dg \, \log (\dg -1)}}{\E{\dg}} \right \}\\
&\geq \E{\dg} (\E{\dg}-1)^{r-1}.
\end{align*}
\end{theorem}
The lower bound of $\E{\dg} (\E{\dg}-1)^{r-1}$ follows from Jensen's inequality
applied to the convex function $x\log(x-1)$ for $x \geq 2$. Here is a consequence of Theorem \ref{thm:vol};
see \cite{AFH} for a related result on finite graphs.
Let $\umn$ be a unimodular network with no leaves almost surely and $\E{\dg} < \infty$.
Let $NBW_r(G,\rt)$ be the set of non-backtracking walks of length $r$ from the root.
Then $NBW_r(G,\rt)$ is in bijection with $S_r(T_G,\rt)$, hence,
\begin{equation} \label{eqn:nbw}
\E{|NBW_r(G,\rt)|} \geq \E{\dg}  (\E{\dg}-1)^{r-1}.
\end{equation}

\subsection{Outline of the paper}
In Section \ref{sec:pre} we discuss some concepts used in the proofs.
Theorem \ref{thm:cover} is proved in Section \ref{sec:thm1}.
Theorems \ref{thm:spec} and \ref{thm:srwspec} and \ref{thm:vol} are proved in Section \ref{sec:thm2}.
The proof of Theorem \ref{thm:cover} is based on entropy of the non-backtracking walk.
The proofs of Theorems \ref{thm:spec} and \ref{thm:srwspec} make crucial use of the local convergence topology along with spectral methods.

\subsection*{Acknowledgements}
I thank Mikl\'{o}s Ab\'{e}rt and P\'{e}ter Csikv\'{a}ri for helpful comments.

\section{Preliminaries} \label{sec:pre}

\subsection{Spectrum of a unimodular network}

For a unimodular network $\umn$ the quantity $\E{|W_k(G,\rt)|}$ is
the $k$-th moment of a Borel probability measure $\mu_G$ on $\mathbb{R}$,
called its \emph{spectral measure}. Usually, the theory of von Neumann Algebras
is used to define $\mu_G$ is general (see \cite[Section 2.3]{Bor} or \cite[Section 5]{AL}).
One has that
$$\mu_G(B) = \mathbb{E}_{\umn} \left [\mu^{\delta_{\rt}}_{A_G}(B)\right],$$
where $\mu^{\delta_{\rt}}_{A_G}$ is the spectral measure at the function $\delta_{\rt}$
of the adjacency operator of $G$ acting on $\ell^2(G)$.

The spectral radius of $\umn$ can also be formulated in terms of the spectral
measure: $\rho(G) = \sup \,\{|x|: x \in \mathrm{support}(\mu_G)\}.$ The spectral
measure $\mu^{\mathrm{SRW}}_G$ and spectral radius $\rho_{\mathrm{SRW}}(G)$ of the SRW on
$\umn$ are defined similarly with respect to the Markov operator acting on $\ell^2(G)$.
The probability measure $\mu^{\mathrm{SRW}}_{G}$ is supported inside the interval
$[-1,1]$; thus, $\rho_{\mathrm{SRW}}(G) \leq 1$. Moreover, its moments are
$$\int x^k \, d \mu^{\mathrm{SRW}}_G = \E{p_{k}(G,\rt)}.$$

If a sequence of unimodular networks $(G_n,\rt)$ converges to $(G,\rt)$ locally
then their spectral measures $\mu_{G_n}$ converge to $\mu_G$ weakly \cite[Proposition 2.2]{Bor}.
Similarly, $\mu^{\mathrm{SRW}}_{G_n} \to \mu^{\mathrm{SRW}}_{G}$ weakly.

\subsection{Edge rooted graphs and non-backtracking walk}
The non-backtracking walk (NBW) is a Markov process on the space of directed, edge rooted
graphs with no leaves. It does exactly as it sounds.

Define $\mathcal{G}^{**}$ to be set of isomorphism classes of doubly
rooted graphs $(G,x,y)$ analogous to $\mathcal{G}^{*}$. Now for $(G,x,y) \in \mathcal{G}^{**}$
with $(x,y) \in E(G)$, let $e = (x,y), e^{-} = x, e^{+} = y$ and $\cev{e} = (y,x)$. One step of the
non-backtracking walk gives a random element $(G,e^{+},z) \in \mathcal{G}^{**}$, where $z$ is a
uniform random neighbor of $e^{+}$ that is different from $e^{-}$. Let $\nb{G, e}$ denote the
outcome of one step of the NBW starting from $(G,e) = (G,x,y)$. Thus,
\begin{equation*}
\pr{ \nb{G,e} = (H,f)} = \begin{cases}
\frac{1}{\deg(e^{+})-1} & \text{if}\; (H,f) = (G,e^{+},z) \;\text{for}\; z \in B_1(G,e^{+}) \setminus \{e^{-}\} \\
0 & \text{otherwise}.
\end{cases}
\end{equation*}

The NBW on a unimodular network $\umn$ with $\E{\dg} < \infty$ and no leaves almost surely is as follows. 
First, given $\umn$, the random edge rooted network $(G,\rt,\rt')$ \emph{derived from} $\umn$ has the following law.
For every bounded measurable $f : \mathcal{G}^{**} \to \mathbb{R}$,
\begin{equation} \label{eqn:edgerooted}
\E{ f(G,\rt,\rt')} = \frac{\E{\,\sum_{x: x \sim \rt} f(G,\rt,x)}}{\E{\dg}}\,.\end{equation}
The NBW on $\umn$ is the $\mathcal{G}^{**}$-valued process $(G_0,e_0), (G_1,e_1), \ldots$
defined by $(G_0,e_0) = (G,\rt,\rt')$ and $(G_n,e_n) = \nb{G_{n-1},e_{n-1}}$.


The network $(G,\rt,\rt')$ can roughly be thought of as choosing the root of $G$
according to a degree bias from the distribution of $(G,\rt)$, and then
choosing $\rt'$ as a uniform random neighbour of $\rt$. If $(G,\rt)$ is a fixed finite
graph with a uniform random root $\rt$ then $(G,\rt,\rt')$ is rooted at a uniform
random directed edge of $G$.

Also, for a random edge rooted network $(G,e) = (G,e^{-},e^{+}) \in \mathcal{G}^{**}$, we define
its \emph{reversal} $(G, \cev{e})$ as the random edge rooted network whose law satisfies the
following for all bounded measurable $f : \mathcal{G}^{**} \to \mathbb{R}$:
$$\E{f(G,\cev{e})} = \mathbb{E}_{(G,e)} \left[ f(G,e^{+},e^{-})\right].$$

\begin{lemma}[Stationarity of NBW] \label{lem:nbw}
Let $\umn$ be a unimodular network with no leaves almost surely and
satisfying $\E{\dg} < \infty$. Let $(G_0,e_0), (G_1,e_1), \ldots$ be the NBW
on $\umn$. Then the reversal $(G,\cev{e}_0)$ has the same law as $(G,e_0)$, and
each $(G_n,e_n)$ has the same law as $(G_0,e_0)$.
\end{lemma}

\begin{proof}
If $f: \mathcal{G}^{**} \to [0,\infty)$ is measurable then
$$ \E{f(G_0,\cev{e}_0)} = \frac{\,\E{ \sum_{x: x \sim \rt} f(G,x,\rt)}}{\E{\dg}} =
\frac{\E{ \,\sum_{x: x \sim \rt} f(G,\rt,x)}}{\E{\dg}} = \E{f(G_0,e_0)},$$
where the second equality uses the mass transport principle \eqref{eqn:MTP}.
This shows that $(G_0, \cev{e}_0)$ has the same law as $(G_0,e_0)$.

For the second claim, it suffices to show that $(G_1,e_1)$ has the same law as $(G_0,e_0)$.
For $f$ as above, we see from the definition of a NBW step that
\begin{align}
	\nonumber \E{f(G_1,e_1)} &= \mathbb{E}_{(G,\rt,\rt')} \left[ \sum_{z \sim \rt', \,z \neq \rt } \frac{f(G,\rt',z)}{\deg(\rt') - 1} \right] \\
	\label{eqn:stat}
	&= \dfrac{\mathbb{E}_{(G,\rt)} \left [ \sum_{x, z \in V(G)}
		\frac{f(G,x,z)}{\deg(x) - 1} \ind{z \neq \rt, z \sim x, x \sim \rt} \right ]}{\E{\dg}}.
\end{align}
The function $F : \mathcal{G}^{**} \to [0, \infty)$ defined by
$$F(G,y,z) = \sum_{x \in V(G)} \frac{f(G,x,z)}{\deg(x) - 1} \ind{z \neq y, x \sim z, x \sim y}$$
is an isomorphism invariant. The mass transport principle applied to it gives
$$\E{ \sum_{z \in V(G)} F(\rt,z)} = \E{ \sum_{z \in V(G)} F(z,\rt)}.$$
The l.h.s.~above is the numerator of (\ref{eqn:stat}). The term inside $\E{\cdot}$ on the r.h.s.~is
\begin{align*}
	\sum_{x, z \in V(G)} \frac{f(G,x,\rt)}{\deg(x) - 1} \ind{z \neq \rt, x \sim \rt, x \sim z} = &
	\sum_{x \in V(G)} \frac{f(G,x,\rt)}{\deg(x) - 1} \sum_{z \in V(G)} \ind{z \neq \rt, x \sim \rt, x \sim z} \\
	& = \sum_{x: x \sim \rt} f(G,x,\rt).
\end{align*}
Therefore, $\E{\dg} \cdot \E{f(G_1,e_1)} =  \E{\sum_{x: x \sim \rt} f(G,x,\rt)} = \E{\sum_{x: x \sim \rt} f(G,\rt,x)}$.
This proves $\E{f(G_1,e_1)} = \E{f(G_0,e_0)}$.
\end{proof}

\subsection{Entropy}
We mention some concepts of Shannon entropy that we will use; for a reference see \cite{CT}.
Let $X$ be a random variable with values in a countable state space $\Omega$. If $p(x)$ is
the probability density of $X$ then the entropy of $X$ is
$$H[X] = \sum_{x \in \Omega} - p(x) \log p(x) = \mathbb{E}_{X} \left [ - \log p(X)\right].$$
Let $(X,Y)$ be jointly distributed on $\Omega^2$ and let $p(y | x)$ be the conditional density
of $Y$ given $\{X = x\}$ ($p(y|x) \equiv 0$ if $p(x) = 0$). The conditional
entropy of $Y$ given $X$ is
$$H [ Y | X] = \mathbb{E}_{X} \left [ \sum_{y \in \Omega} - p(y|X) \log p(y|X) \right].$$
If $H[X,Y]$ and $H[X]$ are both finite then $H[Y|X] = H[X,Y] - H[X]$. If $Y$ is measurable
with respect to $X$ then $H[Y|X] = 0$. If $(X,Y,Z)$ are jointly distributed such that
$Y$ is conditionally independent of $Z$ given $X$ then $H[Y | X, Z] = H[Y | X]$.
If $(X_0, \ldots, X_n)$ are jointly distributed then the chain rule of entropy states
$$H[X_0,\ldots,X_n] = H[X_0] + H[X_1 | X_0] + H[X_2 | X_1,X_0] + \cdots + H[X_n | X_{n-1}, \ldots, X_0].$$

\paragraph{\textbf{Entropy of the NBW step.}}
If $(G,x,y) \in \mathcal{G}^{**}$ is edge rooted without leaves then $H[\nb{G,x,y} \mid (G,x,y)] = \log(\deg(y) -1)$.
This implies that if $(G,\rt,\rt')$ is a random edge rooted graph without leaves, almost surely, then
$H[\nb{G,\rt,\rt'} \mid (G,\rt,\rt')] = \mathbb{E}_{(G,\rt,\rt')}[\log (\deg(\rt') -1)]$. In particular, if $(G,\rt,\rt')$ is derived
from a unimodular network $\umn$ via (\ref{eqn:edgerooted}), then the edge reversal invariance
of $(G,\rt,\rt')$ (Lemma \ref{lem:nbw}) applied to $\mathbb{E}_{(G,\rt,\rt')}[\log (\deg(\rt') -1)]$ gives the entropy of a
NBW step on a unimodular network:
\begin{equation} \label{eqn:NBWentropy}
H[\nb{G,\rt,\rt'} \mid (G,\rt,\rt')] = \mathbb{E}_{(G,\rt,\rt')} \left [ \log (\dg -1)\right] = \frac{\E{\dg \log (\dg -1)}}{\E{\dg}}.
\end{equation}

\section{Spectral radius of unimodular trees} \label{sec:thm1}

In order to prove Theorem \ref{thm:cover} we will consider unimodular networks with
edge weights and bound the expectation of weighted closed walks. By choosing appropriate
weights we will deduce both statements in Theorem \ref{thm:cover}. Let $(T,x) \in \G$ be
a tree. Let $w \in W_{2k}(T,x)$ and let the sequence of vertices visited by $w$ be denoted
$w_0 = x, w_1, \ldots, w_{2k} = x$. Let $e_j = (w_{j-1},w_j)$. The \emph{height profile} of $w$
is the function $h_{w}: \{0,1, \ldots, 2k\} \to \{0,1,2,\ldots\}$ defined by $h_w(j) = \mathrm{dist}_{T}(x, w_j)$.
The height profile is a Dyck path of length $2k$. The \emph{forward steps} of $w$ is the sequence of $k$
directed edges $e_{j_1}, \ldots, e_{j_k}$ for which $h_w(j_i) - h_w(j_i-1) = 1$, and such a $j_i$ is a \emph{forward time}.
The walk $w$ is uniquely determined by its height profile and forward steps.

Let $c : \mathcal{G}^{**} \to [0,\infty)$ be a weight function such that for some $\delta > 0$
if $(G,x,y)$ is rooted at an edge $(x,y)$ then $c(G,x,y) \geq \delta$.
The weighted number of closed walks of length $2k$ in $(T,x)$ is defined as
$$W_{2k}(T,x,c) = \sum_{w \in W_{2k}(T,x)} \, \prod_{i=1}^{2k} \,c(T,e_i).$$
We will write $c(G,x,y)$ as $c(x,y)$ when there is no confusion.

Define the symmetric weight function $\ka(x,y) = c(x,y) c(y,x)$. Note that if $w$ is a closed walk on a tree then for
every forward step $e_i$ of $w$ there is a unique accompanying step $e_j$ in the reverse direction
to $e_i$ at some time $j > i$. Indeed, $j$ is the first time $w$ traverses the reversal of $e_i$ after time $i$.
Pairing up every forward step with its accompanying reversal we see that
$$W_{2k}(T,x,c) = \sum_{w \in W_{2k}(T,x)} \; \prod_{i \,\text{forward time of}\,w} \ka(e_i).$$

Let $\mathrm{Dyck}(k)$ denote the set of all Dyck paths of length $2k$, which are the set of all
possible height profiles of walks in $W_{2k}(T,x)$. For a neighbour $y$ of $x$, let $W_{2k}(T,x,y,h,c)$
be the weighted sum over all walks in $W_{2k}(T,x)$ whose first step is towards $y$ and which has
height profile $h$, except without accounting for the first weighted step:
$$W_{2k}(T,x,y,h,c) = \sum_{\substack{w \in W_{2k}(T,x) \\ w_1 = y,\, h_w = h}} \;
\prod_{\substack{\text{forward times}\, i, \\ i > 1}} \ka(e_i).$$
Conditioning on the height profile and the first step of a walk gives
\begin{equation} \label{eqn:walkidentity}
W_{2k}(T,x,c) = \sum_{h \in \mathrm{Dyck}(k)} \; \sum_{y: y \sim x} \ka(x,y) W_{2k}(T,x,y,h,c).
\end{equation}

\begin{proposition} \label{prop:1}
Let $(T,\rt)$ be a unimodular tree with finite expected degree and no leaves almost surely.
Recall the edge rooted tree $(T,\rt,\rt')$ derived from $(T,\rt)$ via (\ref{eqn:edgerooted}).
If $h \in \mathrm{Dyck}(k)$, then
$$\E{W_{2k}(T,\rt, \rt', h, c)} \geq \Big \{ (k-1) H \left[\nb{T,\rt,\rt'} \mid (T,\rt,\rt') \right] + 2(k-1) \E{\log c(T,\rt,\rt')}  \Big \}.$$
\end{proposition}

\paragraph{\textbf{Proof of Proposition \ref{prop:1}}}
Jensen's inequality implies
\begin{equation} \label{eqn:prop1}
\E{W_{2k}(T,\rt, \rt', h, c)} \geq \exp \left \{ \E{\log W_{2k}(T,\rt,\rt',h,c)}\right \} .
\end{equation}
Let $(T,x,y) \in \mathcal{G}^{**}$ be an edge rooted tree with no leaves.
We define a probability distribution on the set $\{ w \in W_{2k}(T, x):w_1=y, h_w = h\}$.
Every element of this set is encoded as a sequence of edge rooted trees
$(T_1,e_1)$, $\ldots$, $(T_{2k},e_{2k})$, where $(T_1,e_1) = (T, x,y)$ and $(T_i,e_i)$ is
obtained from $(T_{i-1},e_{i-1})$ by moving along the $i$-th edge of the walk.
Therefore, consider the following probability distribution
$(T_1,f_1), \ldots, (T_{2k},f_{2k})$ on the set.

First, $(T_1,f_1) =  (T,x,y)$. Now consider a stack $S$ of forward times of $h$ that is initialized to $S = [1]$.
For $i > 1$, if $i$ is a forward time then set $(T_{i},f_{i}) = \nb{T_{i-1},f_{i-1}}$
and append $i$ to $S$ by updating $S = [S, i]$. If $i$ is a backward time, let $\ell$ be
the last element of $S$ and set $(T_i,f_i) = (T_{\ell}, \cev{f}_{\ell})$, that is, the reversal of $(T_{\ell}, f_{\ell})$.
Then update $S$ by removing $\ell$ from the end of $S$.
Figure \ref{fig:treewalk} provides an illustration.

Observe that the walk is at the root whenever $S$ in empty and then the next step is a forward step.
The stack $S$ is determined from $h$ and non random. Note that at a forward time $i$,
$(T_i,f_i)$ is conditionally independent of $(T_1,f_1), \ldots, (T_{i-2},f_{i-2})$ given
$(T_{i-1},f_{i-1})$ due to the Markov property of the NBW. During a backward
time $i$, $(T_i,f_i)$ is a (measurable) function of the history $(T_1,f_1), \ldots, (T_{i-1},f_{i-1})$.

\begin{figure}[hbtp]
\begin{center}
\includegraphics[scale=0.5]{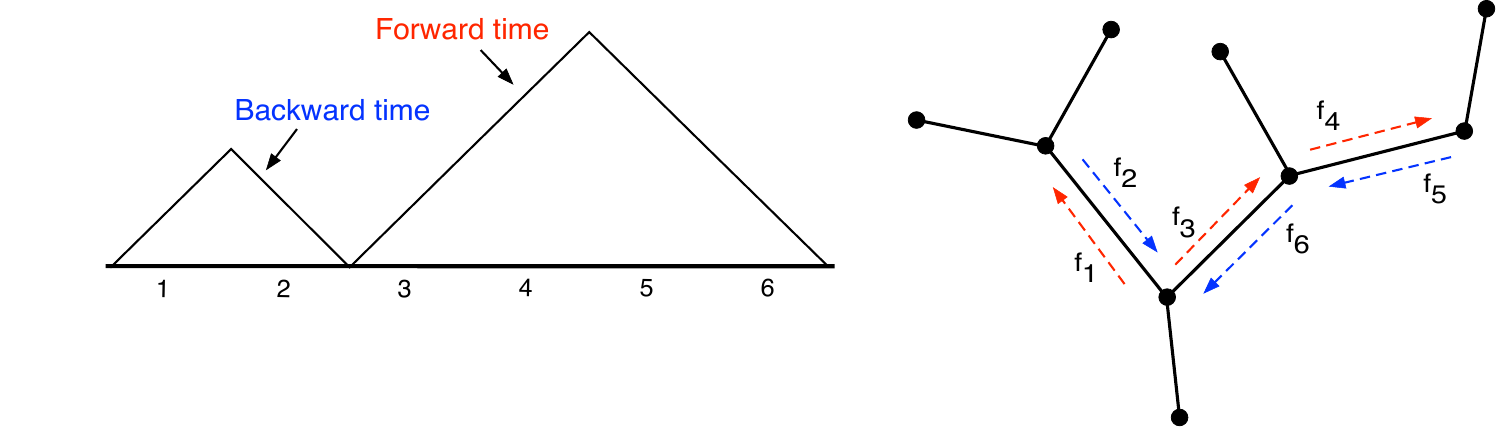}
\caption{A 6-step height profile and a closed walk on the tree associated to it.
Steps 3 and 4 each have two possible choices for a forward step.
The stack $S$ updates as $[1] \to [\,] \to [3] \to [3,4] \to [3] \to [\,]$.}
\label{fig:treewalk}
\end{center}
\end{figure}

\begin{lemma} \label{lem:weightentropy}
Let $(T,x,y)$, $h$ and $(T_1,f_1),\ldots,(T_{2k},f_{2k})$ be as above. Then,
$$
\log W_{2k}(T,x,y,h,c) \geq \sum_{\substack{\text{forward times} \,i,\\ i > 1}} H\left [ (T_i,f_i) \,|\, (T_{i-1},f_{i-1})\right ]
+ \E{\log \ka(T_i,f_i)}.
$$
\end{lemma}

\begin{proof}
For two probability distributions of a countable set $\Omega$ with densities $p$ and $q$, the
Kullback-Leibler Divergence of $p$ from $q$ is
$D( p || q) = \sum_{\omega \in \Omega} \log \big(\frac{p(\omega)}{q(\omega)}\big)\,p(\omega)$.
The divergence is nonnegative, which gives
$$\sum -\log (q(\omega)) \, p(\omega) \geq \sum - \log(p(\omega)) \,p(\omega).$$
If $q$ has the form $q(\omega) = e^{E(\omega)}/Z$, then we get
$\log Z \geq H[X] + \E{E(X)}$, where $X$ is a random variable with probability density $p$.

We apply this to $\Omega = \{ w \in W_{2k}(T, x):w_1=y, h_w = h\}$, X being the
process $(T_1,f_1),\ldots,(T_{2k},f_{2k})$, and $E(w) = \sum_{\text{forward time}\,i > 1} \log \ka(e_i)$
for a walk $w \in \Omega$. We deduce that
$$\log W_{2k}(T,x,y,h,c) \geq H \left [(T_1,f_1), \ldots, (T_{2k},f_{2k}) \right] +
\sum_{\substack{\text{forward times} \,i,\\ i > 1}} \E{\log \ka(T_i,f_i)}.$$

We use the chain rule to calculate $H[(T_1,f_1),\ldots, (T_{2k},f_{2k})]$.
Note that $H[(T_1,f_1)]$ equals 0 because $(T_1,f_1)$ is non random. Therefore,
$$H\left [(T_1,f_1), \ldots, (T_{2k},f_{2k})\right ] =
\sum_{i=2}^{2k} H \left [ (T_i,f_i) \mid (T_{i-1},f_{i-1}), \ldots, (T_1,f_1) \right].$$
During a backward time $i$, $H[(T_i,f_i) \mid (T_{i-1},f_{i-1}) \ldots (T_1,f_1)] = 0$
because $(T_i,f_i)$ is determined from $(T_1,f_1), \ldots, (T_{i-1},f_{i-1})$ and the stack $S$.
At a forward time $i > 1$, the conditional independence of $(T_i,f_i)$ from
$(T_1,f_1), \ldots, (T_{i-2},f_{i-2})$ given $(T_{i-1},f_{i-1})$ implies
$$ H[(T_i,f_i) \mid (T_{i-1},f_{i-1}), \ldots, (T_1,f_1)] = H[(T_i,f_i) \mid (T_{i-1},f_{i-1})].$$
Therefore,
\begin{equation*} H\left [(T_1,f_1), \ldots, (T_{2k},f_{2k}) \right ] =
 \sum_{\substack{i\;\text{forward time} \\ i > 1}} H[(T_i,f_i) \mid (T_{i-1},f_{i-1})]. \qedhere
 \end{equation*}
\end{proof}

Let $(T_1,\rt_1,\rt'_1), \ldots, (T_{2k},\rt_{2k},\rt'_{2k})$ be the law of the process
$(T_1,f_1), \ldots, (T_{2k},f_{2k})$ started from the random edge rooted graph $(T,\rt,\rt')$.
Applying Lemma \ref{lem:weightentropy} to $(T,\rt,\rt')$ and taking expectation over $(T,\rt,\rt')$ gives
$$\E{\log W_{2k}(T,\rt,\rt',h,c)} \geq \sum_{\substack{\text{forward times} \,i,\\ i > 1}}
H [(T_i,\rt_i,\rt'_i) | (T_{i-1},\rt_{i-1},\rt'_{i-1})] + \E{\log \ka(T_i,\rt_i,\rt'_i)}.$$
We claim that every $(T_i,\rt_i,\rt'_i)$ has the law of $(T, \rt, \rt')$. This is certainly the case for $i=1$.
Assume that this is the case for each of the graphs $(T_1,\rt_1,\rt'_1), \ldots, (T_{i-1},\rt_{i-1},\rt'_{i-1})$.
Then $(T_i,\rt_i,\rt'_i)$ either has the law of the tree $\nb{T_{i-1},\rt_{i-1},\rt'_{i-1}}$, or the reversal of one of
$(T_1,\rt_1,\rt'_1), \ldots, (T_{i-1},\rt_{i-1},\rt'_{i-1})$. By Lemma \ref{lem:nbw}, both these operations preserve
the law of $(T,\rt,\rt')$. So the claim follows by induction.

Consequently, for every $i$,
\begin{align} \label{eqn:entropyexp}
H[(T_i,\rt_i,\rt'_i) \mid (T_{i-1},\rt_{i-1},\rt'_{i-1})] &= H[ \nb{(T,\rt,\rt')} \mid (T,\rt,\rt')], \\ \nonumber
\E{\log \ka(T_i,\rt_i,\rt'_i)} & = \E{\log \ka(T,\rt,\rt')}.
\end{align}
As there are $k-1$ forward times $i > 1$, we combine \eqref{eqn:entropyexp} with \eqref{eqn:prop1} to conclude that
\begin{equation*}
\E{W_{2k}(T,\rt,\rt',h,c)}  \geq \exp \big \{ (k-1) H[\nb{T,\rt,\rt'} \mid (T,\rt,\rt')] + (k-1) \E{\log \ka(T,\rt,\rt')} \big \}.
\end{equation*}
The edge reversal invariance of $(T,\rt,\rt')$ implies $\E{\log \ka(T,\rt,\rt')} = 2 \E{\log c(T,\rt,\rt')}$.
This completes the proof of Proposition \ref{prop:1}.

Theorem \ref{thm:cover} is proved using Proposition \ref{prop:1} as follows. Since $\ka(G,x,y) \geq \delta^2$ for
every edge rooted graph $(G,x,y)$, \eqref{eqn:walkidentity} implies
\begin{equation*}
\E{W_{2k}(T,\rt,c)} \geq \delta^2 \, |\mathrm{Dyck}(k)| \, \E{\dg} \E{W_{2k}(T,\rt,\rt' c)}.
\end{equation*}
The number of Dyck paths of length $2k$ is the Catalan number $\frac{1}{k+1} \binom{2k}{k}$.
It is easily seen that $|\mathrm{Dyck}(k)|^{1/2k} \to 2$ as $k \to \infty$. Proposition \ref{prop:1} thus implies
\begin{equation} \label{eqn:weightedspectralradius}
\liminf_{k \to \infty} \,\E{|W_{2k}(T,\rt,c)}^{1/2k} \geq
2 \exp \left \{ \frac{1}{2}H[\nb{T,\rt,\rt'} | (T,\rt,\rt')] + \E{\log c(T,\rt,\rt')} \right \}.
\end{equation}
Plugging the expression for $H[\nb{T,\rt,\rt'} \mid (T,\rt,\rt')]$ from (\ref{eqn:NBWentropy}),
and setting $c(G,x,y) \equiv 1$ in (\ref{eqn:weightedspectralradius}), provides the first
lower bound to $\rho(T)$ stated in Theorem \ref{thm:cover}. If $(T,\rt)$ has degrees bounded by $\Delta$
almost surely, then the first lower bound to $\rho_{\mathrm{SRW}}(T)$ stated in Theorem \ref{thm:cover}
follows from \eqref{eqn:weightedspectralradius} by having $c(G,x,y) = 1/\mathrm{deg}_G(x)$ and $\delta = 1/\Delta$.

The second group of lower bounds in Theorem \ref{thm:cover} are derived
from convexity. Jensen's inequality applied to $x \to x \log(x-1)$ for $x \geq 2$
gives $$\E{\dg \log(\dg-1)} \geq \E{\dg}\log(\E{\dg}-1),$$ which provides the second
lower bound to $\rho(T)$. Jensen's inequality applied to $x \to e^x$ for the probability measure
$f \to \E{\dg f}/\E{\dg}$ gives
$$\exp \left \{ \frac{\E{\dg \log \dg}}{\E{\dg}}\right \} \leq \frac{\E{\dg^2}}{\E{\dg}}.$$
Taking reciprocals above in combination with the bound
$$\E{\dg \log(\dg-1)} \geq \E{\dg}\log(\E{\dg}-1)$$ provides the second
stated lower bound to $\rho_{\mathrm{SRW}}(T)$.
\qed

\section{Alon-Boppana bound and volume growth: proofs of Theorems \ref{thm:spec}, \ref{thm:srwspec} and \ref{thm:vol}} \label{sec:thm2}

\subsection{Proof of Part I of Theorem \ref{thm:spec}}

Since $\mu_{G_n} \to \mu_{G}$ weakly, we have $\liminf_n \mu_{G_n}(|x| > a) \geq \mu_{G}(|x| > a)$ for every $a$.
Therefore, since $\rho(G) < \infty$, Lemma \ref{lem:spectralmass} below implies that
\begin{equation} \label{eqn:0}
\liminf_n \, \mu_{G_n} \left ( \{ |x| > \rho(T_G) - \eps\} \right) \geq
\frac{\E{|W_{2k}(T_{G},\rt)|} - (\rho(T_G)-\eps)^{2k}}{\rho(G)^{2k}}\;\;\text{for every}\;\;k.
\end{equation}
Since $\E{|W_{2k}(T_G,\rt)|}^{1/2k} \to \rho(T_G)$ as $k \to \infty$,
we may choose a large $K$ such that $\E{|W_{2K}(T_G,\rt)|} \geq (\rho(T_G) - \frac{\eps}{2})^{2K}$.
Then, by defining
$$c(\eps, \rho(G), \rho(T_G)) = \frac{(\rho(T_G) - \frac{\eps}{2})^{2K} - (\rho(T_G) - \eps)^{2K}}{\rho(G)^{2K}},$$
the inequality \eqref{eqn:0} applied to $k := K$ implies that $\liminf_n \, \mu_{G_n} ( \{ |x| > \rho(T_G) - \eps\} )
\geq c(\eps, \rho(G), \rho(T_G))$. This completes the proof of part I of Theorem \ref{thm:spec}. \qed

\begin{lemma} \label{lem:spectralmass}
Let $(H,\rt)$ be a unimodular network with $\rho(H) < \infty$.
For $0 < a < \rho(T_H)$ and any $k \geq 0$ we have
$$\mu_H \left ( \{ |x| > a \}\right) \geq 
\frac{ \E{|W_{2k}(T_H,\rt)|} - a^{2k}}{\rho(H)^{2k}}.$$
\end{lemma}

\begin{proof}
Let $\nu = \mu_{H}(\{|x| > a \})$.
The moments of the spectral measure of $(H,\rt)$ satisfy
$$\int x^{2k} \,d \mu_{H} = \E{|W_{2k}(H,\rt)|} \geq \E{|W_{2k}(T_H,\rt)|}.$$
On the other hand, we may bound the moments from above as follows. Note
that $\mu_{H}(\{|x| > \rho(H)\}) = 0$ by definition of the spectral radius. Therefore,
\begin{align*}
\int x^{2k} \,d \mu_{H} &= \int_{|x| \leq a} x^{2k} \,d \mu_{H} + \int_{|x| > a} x^{2k} \,d \mu_{H}\\
&\leq a^{2k} \,\mu_H(\{|x| \leq a\}) + \rho(H)^{2k} \,\mu_H(\{|x| > a \})\\
&= a^{2k} + \nu \left( \rho(H)^{2k} - a^{2k} \right).
\end{align*}

Combining the lower and upper bounds on the moments we get that for every $k$,
\begin{equation*} \nu \geq \frac{\E{|W_{2k}(T_H,\rt)|} - a^{2k}}{\rho(H)^{2k} - a^{2k}}
\geq \frac{\E{|W_{2k}(T_H,\rt)|} - a^{2k}}{\rho(H)^{2k}}. \qedhere
\end{equation*}
\end{proof}

\subsection{Proof of Part II of Theorem \ref{thm:spec}}

\begin{lemma} \label{lem:core}
Let $G$ be a finite and connected graph with 2-core $G^{\mathrm{core}}$; recall it is obtained
by iteratively removing leaves from $G$ until a subgraph with no leaves remains.
If $G$ is not a tree then $d_{\mathrm{av}}(G^{\mathrm{core}}) \geq d_{\mathrm{av}}(G)$.
Moreover, $\sigma_j(G) \geq \sigma_j(G^{\mathrm{core}})$, where $\sigma_j(H) =0$ by convention if $j > |H|$.
(Recall $\sigma_j(H)$ is the $j$-th largest eigenvalue of $H$ in absolute value counted with multiplicity).
\end{lemma}

\begin{proof}
Since $G$ is not a tree, $|E(G)| \geq |G|$. If $G'$ is obtained from $G$ by removing a leaf then
$d_{\mathrm{av}}(G') = 2(|E(G)|-1)/(|G|-1) \geq d_{\mathrm{av}}(G)$ since $|E(G)| \geq |G|$.
Moreover, the adjacency matrix of $G'$ is a principal minor of the adjacency matrix of $G$.
Suppose $\lambda_1 \geq \lambda_2 \geq \cdots \geq \lambda_{n}$ are the $n = |G|$ eigenvalues of $G$,
and $\nu_1 \geq \nu_2 \geq \cdots \geq \nu_{n-1}$ are the eigenvalues of $G'$.
From the Cauchy interlacing theorem we have $\lambda_1 \geq \nu_1 \geq \lambda_2 \geq \nu_2 \geq
\cdots \geq \nu_{n-1} \geq \lambda_n$. This implies that $\sigma_j(G) \geq \sigma_j(G')$ for every $j$.

The observations above imply $d_{\mathrm{av}}(G^{\mathrm{core}}) \geq d_{\mathrm{av}}(G)$
and $\sigma_j(G) \geq \sigma_j(G^{\mathrm{core}})$.
\end{proof}

We now prove part II of the theorem. Let $G_{n_i}$ be a subsequence
such that $\liminf_n \, \sigma_j(G_n) = \lim_i \, \sigma_j(G_{n_i})$. Clearly,
$\liminf_i \, 2 \sqrt{d_{\mathrm{av}}(G_{n_i})-1} \geq \liminf_n\, 2 \sqrt{d_{\mathrm{av}}(G_n)-1}$.
Therefore, it is enough to show that
$\liminf_i\, \sigma_j(G_{n_i}) \geq \liminf_i \, 2 \sqrt{d_{\mathrm{av}}(G_{n_i})-1}$.
Henceforth, we denote the subsequence $G_{n_i}$ as $G_n$ and $\sigma_j = \lim_i \, \sigma_j(G_{n_i})$.
In the new notation, we must show that
\begin{equation} \label{eqn:1}
\sigma_j \geq \liminf_n\, 2 \sqrt{d_{\mathrm{av}}(G_n)-1}\,.
\end{equation}

First, suppose it is the case that for an infinite subsequence $G_{n_k}$ of $G_n$ we have that
$|G^{\mathrm{core}}_{n_k}| \to \infty$. It suffices to show that
$\sigma_j \geq \liminf_k \, 2\sqrt{d_{\mathrm{av}}(G_{n_k})-1}$ because the latter limit infimum
is an upper bound to $\liminf_n \, 2\sqrt{d_{\mathrm{av}}(G_{n})-1}$. Let us denote
the subsequence $G_{n_k}$ as $H_n$. Thus, we must show that
\begin{equation} \label{eqn:2}
\sigma_j \geq \liminf_n\, 2 \sqrt{d_{\mathrm{av}}(H_n)-1}\,.
\end{equation}

The graphs $H^{\mathrm{core}}_{n}$ are connected, have no leaves and have maximum
degree at most $\Delta$. If $\rt_n$ is a uniform random root of $H^{\mathrm{core}}_n$ then
the unimodular networks $(H^{\mathrm{core}}_n,\rt_n)$ have a subsequential limit $\umn$.
Indeed, the subset of $\G$ consisting of rooted isomorphism classes of graphs of maximal degree
$\Delta$ is compact because there are at most $\Delta^r$ possibilities for the $r$-neighbourhood
of the root of such graphs. Prokhorov's theorem states that Borel probability measures on a compact
metric space is compact in the weak topology. This provides a subsequential limit of
$(H^{\mathrm{core}}_n,\rt_n)$ in the local topology.

Let us reduce to a convergent subsequence $(H^{\mathrm{core}}_{n_i},\rt_{n_i})$, converging to $\umn$.
Let $(T,\rt)$ be the universal cover of $\umn$. Then $(T,\rt)$ has no leaves and has maximum degree at
most $\Delta$ almost surely because $\umn$ inherits these properties from the sequence $H^{\mathrm{core}}_{n_i}$.
Part I of the theorem implies for every $\eps > 0$,
$$\liminf_i \, \mu_{H^{\mathrm{core}}_{n_i}}( \{|x| > \rho(T) - \eps\}) > 0.$$
Since $|H^{\mathrm{core}}_{n_i}| \to \infty$ by assumption, $\sigma_j(H^{\mathrm{core}}_{n_i}) \geq \rho(T) - \eps$
for all large $i$ due to the bound above. From Theorem \ref{thm:cover} we
have $\rho(T) \geq 2 \sqrt{\E{\deg(\rt)}-1} = \lim_i \, 2 \sqrt{ d_{\mathrm{av}}(H^{\mathrm{core}}_{n_i}) -1}$.
Therefore, since $\eps$ is arbitrary,
\begin{equation} \label{eqn:3}
\liminf_i \, \sigma_j(H^{\mathrm{core}}_{n_i}) \geq \liminf_i \, 2 \sqrt{ d_{\mathrm{av}}(H^{\mathrm{core}}_{n_i}) -1}\,.
\end{equation}

Lemma \ref{lem:core} implies $\sigma_j(H_{n_i}) \geq \sigma_j(H^{\mathrm{core}}_{n_i})$.
Taking limit infimum in $i$ implies
\begin{equation} \label{eqn:4} \sigma_j \geq \liminf_i \sigma_j(H^{\mathrm{core}}_{n_i}). \end{equation}
Indeed, $\sigma_j$ is the limit of $\sigma_j(H_{n_i})$ because $H_{n_i}$ is a subsequence of $G_n$
and $\sigma_j(G_n)$ converges to $\sigma_j$ by assumption. Lemma \ref{lem:core} also implies that
\begin{equation} \label{eqn:5}
2 \, \sqrt{ d_{\mathrm{av}}(H^{\mathrm{core}}_{n_i}) -1} \geq 2 \, \sqrt{ d_{\mathrm{av}}(H_{n_i}) -1}\,.
\end{equation}
The required inequality in \eqref{eqn:2} follows by combining the inequality in \eqref{eqn:4} with
the one from \eqref{eqn:3}, followed by the inequality in \eqref{eqn:5}.

We are left to consider the case where the core graphs of the sequence $G_n$ have bounded size,
possibly being empty. Due to compactness, as explained above, the unimodular networks $(G_n,\rt_n)$,
where $\rt_n$ is a uniform random root of $G_n$, have a subsequential limit $\umn$.
We claim that $\umn$ is an infinite unimodular tree of expected degree $2$.

Indeed, $\umn$ is infinite almost surely because $G_n$ is connected and $|G_n| \to \infty$.
To see that $\umn$ is a tree observe that the graph induced on $G_n \setminus G^{\mathrm{core}}_n$
contains no cycles. Thus $B_r(G_n,\rt_n)$ is a tree so long as $\rt_n$ is not within distance $r$ of
$G^{\mathrm{core}}_n$, and this happens with probability at least $1 - \frac{|G^{\mathrm{core}}_n| \Delta^r}{|G_n|} \to 1$.
This implies that the finite neighbourhood sampling statistics of $\umn$ are supported on trees, and thus, $\umn$ is a tree.

Now we argue that $\umn$ has expected degree 2.
Suppose $l_n$ is the number of vertices removed from $G_n$ during the leaf peeling
procedure that generates $G^{\mathrm{core}}_n$. Then $l_n \to \infty$ as $n \to \infty$
because $|G^{\mathrm{core}}_n|$ remains bounded. Moreover,
$$|G_n| = |G^{\mathrm{core}}_n| + l_n \;\;\text{and}\;\; |E(G_n)| = |E(G^{\mathrm{core}}_n)| +l_n.$$
Therefore,
$$d_{\mathrm{av}}(G_n) = 2 \, \frac{|E(G^{\mathrm{core}}_n)| +l_n}{|G^{\mathrm{core}}_n| +l_n} \, \longrightarrow \, 2,$$
which shows that $\umn$ has expected degree 2 because $d_{\mathrm{av}}(G_n)$ converges to it
due to the graphs $G_n$ having uniformly bounded degrees. 

Now we claim that $\rho(G) \geq 2$. As $\umn$ is infinite, there is an infinite one ended path
starting from $\rt$. Therefore, $|W_{2k}(G,\rt)|$ is at least the number of closed walks of length
$2k$ on an infinite one ended path starting from its initial leaf vertex. This
quantity is the Catalan number $C_k = \frac{1}{k+1}\binom{2k}{k}$. Thus,
$\E{|W_{2k}(G,\rt)|} \geq C_k$ and we conclude that $\rho(G) \geq 2$ because $C_k^{1/2k} \to 2$.

The tree $\umn$ is its own universal cover. Using part I of the theorem and arguing
as before we deduce that $\sigma_j = \lim_n \, \sigma_j(G_n) \geq 2$. On the other hand,
$$\liminf_n \, 2 \sqrt{d_{\mathrm{av}}(G_n) - 1} \leq 2 \sqrt{\mathbb{E}_{\umn}[\deg(\rt)]-1} = 2.$$
These bounds imply the required inequality in \eqref{eqn:1} and completes the proof of part II of the theorem.
\qed

\subsection{Proof of Theorem \ref{thm:srwspec}}
For a finite graph $G$ let us denote
$$ \bar{D}(G) = \frac{2 \sqrt{d_{\mathrm{av}}(G)-1}}{\frac{1}{2|E(G)|} \sum_{x \in G} (\deg{x})^2}.$$
Note that $\bar{D}(G)$ is a continuous function in the topology of local convergence since
$2 |E(G)| = d_{\mathrm{av}}(G) |G|$.

First we shall consider the proof when the sequence of graph $G_n$ has no leaves.
Then given $\eps > 0$, consider a subsequence $G_{n_i}$ such that
$\mu^{\mathrm{SRW}}_{G_{n_i}}(\{ |x| > \bar{D}(G_{n_i}) - \eps\})$ converges to the limit infimum of
$\mu^{\mathrm{SRW}}_{G_n}(\{ |x| > \bar{D}(G_n) - \eps\})$.
Due to compactness, there is a further locally convergent subsequence $(G_{n_{i_j}}, \rt_{n_{i_j}}) \to (G,\rt)$.
It suffices to prove the claim for this convergent subsequence.
Denote the sequence of graphs $G_{n_{i_j}}$ as $H_n$.

Arguing as in the proof of part I of Theorem \ref{thm:spec} we see that
$$\liminf_{n \to \infty} \, \mu^{\mathrm{SRW}}_{H_n} \left ( \left \{ |x| > \rho_{\mathrm{SRW}}(T_G) - \frac{\eps}{2} \right \} \right) > 0, $$
where $\umn$ is the limit. Theorem \ref{thm:cover} applied to its universal cover $T_G$ implies
$$\rho_{\mathrm{SRW}}(T_G) \geq \bar{D}(G,\rt) =: \frac{2 \,\E{\dg} \sqrt{\E{\dg}-1}}{\E{\dg^2}}.$$
Observe that $\bar{D}(H_n) \to \bar{D}(G,\rt)$ because $(H_n,\rt_n)$ converges to $\umn$
and all the graphs are of bounded degree. Thus, for all sufficiently large $n$, we have
$\bar{D}(G,\rt) \geq \bar{D}(H_n) - \frac{\eps}{2}$. For any such $n$,
$$ \mu^{\mathrm{SRW}}_{H_n}(\{ |x| > \bar{D}(H_n) - \eps\}) \geq
\mu^{\mathrm{SRW}}_{H_n}(\{ |x| >  \rho_{\mathrm{SRW}}(T_G) - \frac{\eps}{2}\}).$$
This implies the required claim for the sequence $H_n$ and completes the
proof of the theorem when the sequence $G_n$ has no leaves.

For the general case of $|G^{\mathrm{core}}_n| / |G_n| \to 1$, we will use the following two lemmas.
\begin{lemma} \label{lem:srw1}
Let $G$ be a finite and connected graph with a non-empty 2-core. Then
$$ \mu_G^{\mathrm{SRW}} ( |x| > a) \geq
\mu_{G^{\mathrm{core}}}^{\mathrm{SRW}}  (|x| > a) - 10\log \big (|G|/|G^{\mathrm{core}}|\big).$$
\end{lemma}

\begin{proof}
Suppose that $G$ has $n$ vertices and $|G^{\mathrm{core}}| = |G| - m$. Let $P$ be the
Markov operator of $G$, $D$ the diagonal matrix of vertex degrees, and $A$ the adjacency matrix.
Write $d_x = \mathrm{deg}(x)$ for a vertex $x$ of $G$.	

Observe that $P = D^{-1}A$, so $P = D^{-1/2} (D^{-1/2}A D^{-1/2}) D^{1/2}$.
Hence $P$ has the same eigenvalues as the symmetric matrix $L = D^{-1/2} A D^{-1/2}$.
If $\mu_L$ is the empirical measure for the eigenvalues of $L$, then $\mu^{\mathrm{SRW}}_G(|x| > a) = \mu_L (|x| > a)$.

Let $u$ be a leaf of $G$ with neighbour $v$ and consider the reduced graph $G' = G \setminus \{u\}$.
Let $\hat{L}$ be the submatrix of $L$ obtained by removing the row and column associated to vertex $u$.
By the Cauchy interlacing theorem,
\begin{equation} \label{eqn:srw1a}
\mu_L (|x| > a) \geq \mu_{\hat{L}} (|x| > a) (1- \frac{1}{n}) \geq \mu_{\hat{L}} (|x| > a)  - \frac{1}{n}.
\end{equation}

If $L'$ is the $L$-matrix associated to $G'$ then
$$L'(x,y) = \frac{\ind{x \sim y \,\text{in}\, G}}{\sqrt{(d_x - \ind{x=v})(d_y - \ind{y=v})}} \quad (x,y \neq u).$$
Note $L(x,y) = \ind{x \sim y\, \text{in} \, G}/\sqrt{d_x d_y}$. Consequently, $L' = \hat{L} + E$ where
$$ E(x,y) = \frac{d_y^{-1/2} \ind{x=v, \, y \sim b} + d_x^{-1/2}\ind{y=v, \, x \sim v}}{\sqrt{d_v-1} (\sqrt{d_v} + \sqrt{d_v -1})} \quad (x,y \neq u).$$
The matrix $E$ is symmetric with rank at most 2. Indeed, only the row and column associated to vertex $v$
is non-zero. So by the Weyl interlacing theorem (see \cite{Bor}),
\begin{equation} \label{eqn:srw1b}
\big | \mu_{L'}(|x| > a) - \mu_{\hat{L}}(|x| > a) \big| \leq \frac{2}{n-1}.
\end{equation}

If we combine \eqref{eqn:srw1a} with \eqref{eqn:srw1b} we infer that
$$ \mu_{G}^{\mathrm{SRW}}(|x| > a) \geq \mu_{G'}^{\mathrm{SRW}}(|x| > a) - \frac{3}{n-1}.$$
It follows from iteration that
\begin{align*}\mu_{G}^{\mathrm{SRW}}(|x| > a) & \geq \mu_{G^{\mathrm{core}}}^{\mathrm{SRW}}(|x| > a)
- 3\Big( \frac{1}{n-1} + \cdots + \frac{1}{n-m}\Big) \\
& \geq  \mu_{G^{\mathrm{core}}}^{\mathrm{SRW}}(|x| > a) - 10 \log(|G|/|G^{\mathrm{core}}|). \qedhere
\end{align*}
\end{proof}

\begin{lemma} \label{lem:srw2}
Let $G$ be a finite and connected graph with vertex degrees at most $\Delta$.
Let $\frac{p}{q}$ be a rational number that is not an eigenvalue of the Markov operator of $G$
and suppose that $gcd(p,q) = 1$. There is a constant depending on $q$ and $\Delta$ such that
for $0 < \delta < 1$,
$$\mu_{G}^{\mathrm{SRW}}\Big(\big[\frac{p}{q}-\delta, \frac{p}{q}+\delta \big]\Big) \leq \frac{\mathrm{const}(q,\Delta)}{|\log \delta|}.$$
\end{lemma}

\begin{proof}
Set $\mu = 	\mu_{G}^{\mathrm{SRW}}([(p/q)-\delta, (p/q)+\delta])$, let $P$ denote the Markov operator of $G$,
and suppose $G$ has $n$ vertices. We may assume $|\frac{p}{q}| \leq 2$ for otherwise $\mu$ is zero.

Consider the determinant of $(p/q)I - P$ in two different ways. On the one hand, $P$ has a full set of
eigenvalues inside $[-1,1]$, which implies that
\begin{equation} \label{eqn:srw2a}
\big |\mathrm{det}\Big (\frac{p}{q}I - P \Big) \big | = \prod_{\text{eig.~val.~}\,\lambda}\,  \big | \frac{p}{q} - \lambda \big|
\leq \delta^{n \mu} \, 3^{n(1-\mu)}.
\end{equation}
On the other hand, consider $\ell = lcm (\mathrm{deg}(x); x \in G)$ which is at most $lcm(1, 2, \ldots, \Delta)$.
Now $\mathrm{det}\big((p/q)I-P\big) = (q\ell)^{-n} \mathrm{det}(p\ell I - q\ell P)$,
and the matrix $p\ell I - q\ell P$ has integer entries as well as a non-zero determinant. So $| \mathrm{det}(p\ell I - q\ell P)| \geq 1$.
This implies that
\begin{equation} \label{eqn:srw2b}
\big |\mathrm{det}\Big (\frac{p}{q}I - P \Big) \big |  \geq (q \ell)^{-n}.
\end{equation}

Comparing \eqref{eqn:srw2a} with \eqref{eqn:srw2b} provides the inequality from the lemma.
\end{proof}

To conclude the proof suppose $G_n$ is a sequence of graphs as in the theorem and $\eps > 0$.
It suffices to consider only rational values of $\eps$ with $0 < \eps < 1$.
Due to having bounded degrees and $|G^{\mathrm{core}}_n| / |G_n| \to 1$, it is easy to see that
$|\bar{D}(G_n) - \bar{D}(G^{\mathrm{core}}_n)| \to 0$. Now if $\bar{D}(G_n) \leq \bar{D}(G^{\mathrm{core}}_n)$,
then $\mu_{G_n}^{\mathrm{SRW}}( |x| > \bar{D}(G_n) -\eps) \geq \mu_{G_n}^{\mathrm{SRW}}( |x| > \bar{D}(G^{\mathrm{core}}_n) -\eps)$.

If not, consider a rational number $r$ such that $r$ is not an eigenvalue of the Markov operator of $G_n$,
$r \leq \bar{D}(G^{\mathrm{core}}_n) - \eps$, and $|r - \bar{D}(G_n) + \eps| \leq 2 |\bar{D}(G^{\mathrm{core}}_n) - \bar{D}(G_n)|$.
Since the graphs have degrees bounded by $\Delta$ it is easy to see that the denominator of $r$ remains bounded
in terms of on $\eps$ and $\Delta$. Then by Lemma \ref{lem:srw2} with $\delta = 2 |\bar{D}(G^{\mathrm{core}}_n) - \bar{D}(G_n)|$,
since $r \leq \bar{D}(G^{\mathrm{core}}_n) - \eps \leq \bar{D}(G_n) - \eps \leq r +\delta$, we find that
\begin{equation} \label{eqn:srw2c}
 \mu_{G_n}^{\mathrm{SRW}} (|x| > \bar{D}(G_n) -\eps) \geq
\mu_{G_n}^{\mathrm{SRW}} (|x| > \bar{D}(G^{\mathrm{core}}_n) -\eps) -
\frac{\mathrm{const}(\eps, \Delta)}{| \log (|\bar{D}(G^{\mathrm{core}}_n) - \bar{D}(G_n)|)|}.
\end{equation}

The inequality \eqref{eqn:srw2c} thus holds irrespective of the ordering between $\bar{D}(G_n)$ and
$\bar{D}(G^{\mathrm{core}}_n)$. By Lemma \ref{lem:srw1}, we may replace $\mu_{G_n}^{\mathrm{SRW}}$
by $\mu_{G^{\mathrm{core}}_n}^{\mathrm{SRW}}$ on the right hand side of \eqref{eqn:srw2c} while
incurring a penalty of vanishing order $\log( |G_n|/|G^{\mathrm{core}}_n)$. Then considering the limit
infimum and applying the previous case to $G^{\mathrm{core}}_n$ leads to the theorem. \qed

It may be interesting to see to what extent Theorrem \ref{thm:srwspec} holds when $G^{\mathrm{core}}_n$
only occupies a positive fraction of $G_n$.

\subsection{Proof of Theorem \ref{thm:vol}} \label{sec:thm3}

Let $(T,\rt)$ be a unimodular tree with $\E{\dg} < \infty$ and having no leaves
almost surely. Let $S_r(T, x) = \{v \in V(T): \mathrm{dist}_T(x,v) = r\}$. Recall
the height profile of a walk and the notation $W_{2k}(G,x,y,h,c)$ from Section \ref{sec:thm1} (around \eqref{eqn:walkidentity}).
The vertices in $S_r(T,x)$ are in bijection with walks in $W_{2r}(T,x)$ whose height profile is
the Dyck path consisting of $r$ forward steps followed by $r$ backward steps.
Let $h$ denote this particular height profile. Then, with $c(x,y) \equiv 1$,
$$ |S_r(T,x)| = \sum_{y \sim x} |W_{2r}(T,x,y,h,c)|.$$
Theorem \ref{thm:vol} now follows from Proposition \ref{prop:1}.
\qed

\section{Future directions}
It is shown in \cite{AGV} that if an infinite $d$-regular unimodular network has spectral radius $2\sqrt{d-1}$
then it must be the $d$-regular tree. It is also proved that if a sequence of finite, connected,
$d$-regular graphs $G_n$ converges to the $d$-regular tree locally then
apart from $o(|G_n|)$ short cycles, the smallest cycle in $G_n$ has length of order at least
$\log \log |G_n|$. Little is known about such results for arbitrary unimodular networks.
Suppose a sequence of finite and connected graphs $G_n$ of growing size share a
common universal cover $T$. If the spectral measures of the $G_n$ concentrate on
$[-\rho(T),\rho(T)]$ as $n \to \infty$ then does $G_n$ converge to $T$ locally?

\end{document}